\documentclass{article}
\PassOptionsToPackage{numbers, compress}{natbib}

 \usepackage[main, preprint]{tagds_2025}
 \workshoptitle{Topology, Algebra, and Geometry in Data Science}

\usepackage[utf8]{inputenc} 
\usepackage[T1]{fontenc}    
\usepackage{hyperref}       
\usepackage{url}            
\usepackage{booktabs}       
\usepackage{amsfonts}       
\usepackage{nicefrac}       
\usepackage{microtype}      
\usepackage{xcolor}         
\usepackage{enumerate}
\usepackage{amssymb}
\usepackage{amsfonts}
\usepackage{amsmath}
\usepackage{amsthm}
\usepackage{mathtools}
\usepackage{graphicx}
\usepackage{url}
\usepackage[mathscr]{euscript}
 \let\mathscr\relax
\usepackage[scr]{rsfso}
\usepackage{mathrsfs}
\usepackage{amscd}
\usepackage{url}
\usepackage{subcaption}

\newcommand{\qmod}[1]{\operatorname{QM}[#1]}
\newcommand{\qm}{\operatorname{QM} }

\newcommand{\re}{\mathbb{R}}

\newcommand{\N}{\mathbb{N}}

\newcommand{\bfc}{\mathbf{c}}

\newcommand{\lmd}{\lambda}

\newcommand{\vareps}{\varepsilon}

\def\af{\alpha}

\def\gm{\gamma}
\def\rank{\mbox{rank}}

\newcommand{\sig}{\sigma}
\newcommand{\Sig}{\Sigma}

\newcommand{\st}{\mathrm{s.t.}}
\newcommand{\mt}[1]{\mathtt{#1}}

\newcommand{\reff}[1]{(\ref{#1})}

\newcommand{\mc}[1]{\mathcal{#1}}

\newcommand{\ddd}{,\ldots,}

\newcommand{\bdes}{\begin{description}}
\newcommand{\edes}{\end{description}}

\newcommand{\bal}{\begin{align}}
\newcommand{\eal}{\end{align}}

\newcommand{\bnum}{\begin{enumerate}}
\newcommand{\enum}{\end{enumerate}}

\newcommand{\bit}{\begin{itemize}}
\newcommand{\eit}{\end{itemize}}

\newcommand{\bea}{\begin{eqnarray}}
\newcommand{\eea}{\end{eqnarray}}
\newcommand{\be}{\begin{equation}}
\newcommand{\ee}{\end{equation}}

\newcommand{\baray}{\begin{array}}
\newcommand{\earay}{\end{array}}

\newcommand{\bsry}{\begin{subarray}}
\newcommand{\esry}{\end{subarray}}

\newcommand{\bca}{\begin{cases}}
\newcommand{\eca}{\end{cases}}

\newcommand{\bcen}{\begin{center}}
\newcommand{\ecen}{\end{center}}

\newcommand{\bbm}{\begin{bmatrix}}
\newcommand{\ebm}{\end{bmatrix}}

\newcommand{\bmx}{\begin{matrix}}
\newcommand{\emx}{\end{matrix}}

\newcommand{\bpm}{\begin{pmatrix}}
\newcommand{\epm}{\end{pmatrix}}

\newcommand{\btab}{\begin{tabular}}
\newcommand{\etab}{\end{tabular}}

\newtheorem{theorem}{Theorem}[section]

\newtheorem{example}[theorem]{Example}

\newtheorem{remark}[theorem]{Remark}

\numberwithin{equation}{section}

\title{Learning Polynomial Activation Functions for Deep Neural Networks}

\author{
Linghao~Zhang 
\thanks{Department of Mathematics, 
University of California San Diego, 
9500 Gilman Drive, La Jolla, CA 92093, 
\texttt{liz010@ucsd.edu}}
\And
Jiawang~Nie
\thanks{Department of Mathematics, 
University of California San Diego, 
9500 Gilman Drive, La Jolla, CA 92093, 
\texttt{njw@math.ucsd.edu}}
\And
Tingting~Tang 
\thanks{Department of Mathematics and Statistics, 
San Diego State University Imperial Valley, 
720 Heber Ave, Calexico, CA 92231,
\texttt{ttang2@sdsu.edu}}
}

\begin{document}

\maketitle

\begin{abstract}
Activation functions are crucial for deep neural networks. 
This novel work frames the problem of training neural network with learnable polynomial activation functions as a polynomial optimization problem, 
which is solvable by the Moment-SOS hierarchy. 
This work represents a fundamental departure from the conventional paradigm of training deep neural networks, 
which relies on local optimization methods like backpropagation and gradient descent. 
Numerical experiments are presented to demonstrate the accuracy and robustness of optimum parameter recovery in presence of noises.
\end{abstract}

\section{Introduction}

One essential component for deep neural networks (DNN) is the selection of activation functions, such as sigmoid, ReLU, and their numerous variants~\cite{Glorot11, NairHinton10, Ramachandran17,xu15}. 
In recent years, there is a growing interest in designing adaptive or learnable activation functions to enhance deep neural networks~\cite{Agostinelli14,Apicella21}. 
Among many options, polynomial activation functions offer a particularly appealing choice: 
they are convenient to analyze and have a variety of applications, 
including image generation~\cite{Chrysos21,Chrysos19,karras19}, 
face detection~\cite{huang03}, 
and fuzzy modeling~\cite{park02,Zarandi08}. 
Moreover, when the activation is parameterized as a polynomial with learnable coefficients, the network’s input-output map becomes a polynomial function.
This allows tools from algebraic geometry to be used for a rigorous analysis of network properties, 
including the network's expressive power or the complexity of its training process.
Such work offers a new geometric perspective for understanding deep learning~\cite{kileel2019expressive,kubjas2024geometry}. 

The key novelty of this paper arises from the following realization: 
with polynomial activation functions whose coefficients are to be learned, 
the entire training process of a deep neural network can be formulated as a nonconvex polynomial optimization problem (POP). 
This powerful structural insight allows us to move beyond the limitations of traditional gradient-based optimization. 
While classical methods like stochastic gradient descent (SGD) often guarantees convergence to stationary points, 
the POP formulation opens the door to global optimization techniques.

There are efficient computational methods for solving polynomial optimization.
The core technique is the hierarchy of Moment-SOS relaxations. 
It transforms a nonconvex problem into a linear convex optimization problem with the moment cone that can be solved by a hierarchy of semidefinite programming relaxations.
The dual of the moment cone is the nonnegative polynomial cone, which can be approximated by sum of squares (SOS) type representations. 
The sequence of these primal-dual pairs of relaxations is the called Moment-SOS hierarchy. 
An appealing property of this hierarchy is that almost all polynomial optimization problems can be solved exactly, for both their optimal values and optimizers, by applying Moment-SOS relaxations.
We refer to \cite{HKL20,lasserre2015introduction,Lau09,nie2023moment} for detailed introductions to moment and polynomial optimization.
These methods have been successfully applied to control~\cite{korda14,leb24,Pauwels17}, 
data science optimization~\cite{Nie12,NWY17,NYZZ23,NZ25,ZCN24}, 
tensor computation~\cite{GNY22,Nie17,NWZ23,zhou25}, 
and machine learning tasks~\cite{chen21,huang25}.

In this paper, we consider feedforward neural networks in which the weight matrices are given, 
and all hidden layers apply polynomial activation functions whose coefficients are to be learned. 
We formulate the training process as a polynomial optimization problem and solve it by Moment-SOS relaxations. 

For a $D$-hidden-layer neural network,
let the input and output dimensions be $m_0$ and $m_{D+1}$.
Denote by $m_{\ell}$ the number of hidden neurons in the $\ell$th hidden layer. 
For $\ell = 1 \ddd D+1$, the weight matrices are
$
W_{\ell} \in \re^{m_{\ell} \times m_{\ell - 1}}.
$
For each hidden layer $\ell=1 \ddd D$, 
let $p_{\ell} : \re^1 \to \re^1$ be a polynomial activation function such that
\[
p_{\ell}(t) = c_{\ell, d_{\ell}}t^{d_{\ell}} + \cdots + c_{\ell, 1}t + c_{\ell, 0}, 
\quad c_{\ell, 0}, c_{\ell, 1} \ddd c_{\ell, d_{\ell}} \in \re.
\]
For uniqueness of the coefficient recovery, 
we usually normalize the coefficients such that $c_{1,0} =  \cdots = c_{D-1,0} =1$.
The rest of the coefficients are collected in the vector
\be\label{c}
\bfc  \coloneqq  (c_{1,1} , \ldots, c_{1,d_1} \ddd
c_{D-1,1} , \ldots, c_{D-1,d_{D-1}},
c_{D,0}, c_{D,1} \ddd c_{D, d_D}  ) .
\ee
Given an input $x \in \re^{m_0}$, define
\[
h_0 \coloneqq x, \quad v_{\ell} \coloneqq W_{\ell}h_{\ell-1} \in \re^{m_{\ell}},
\quad h_{\ell} \coloneqq p_{\ell}(v_{\ell}) \in \re^{m_{\ell}}, 
\quad \ell = 1 \ddd D.
\]
In the above, the activation function $p_{\ell}$ is applied element-wise.
The network output is
\[
f(x;\bfc) \coloneqq W_{D+1}h_D \in \re^{m_{D+1}}.
\]
Equivalently, if $P_{\ell} : \re^{m_{\ell-1}} \to \re^{m_{\ell}}$ denotes the $\ell$th layer operator such that
\[
P_{\ell}(u) = \big( p_{\ell}( w_{\ell, 1}^Tu ) \ddd p_{\ell}( w_{\ell, m_{\ell}}^T u ) \big)
\]
with $w_{\ell, j}^T$ denoting the $j$th row of $W_{\ell}$,
then
\be\label{f_map}
f(x;\bfc) = W_{D+1}P_D( P_{D-1} ( \cdots P_1(x)) ).
\ee
Given a training dataset of input-output pairs $\{ (x_i, y_i) \}_{i=1}^N$,
we consider the optimization problem
\be\label{uncont}
\min_{\bfc} \quad \mc{L}(\bfc) \coloneqq \Big\| \frac{1}{N} \sum\limits_{i=1}^N \big( f(x_i; \bfc) - y_i \big) \Big\|_{\infty},
\ee
where $\mc{L}(\bfc)$ is the loss function.
The epigraph reformulation of \reff{uncont} is
\be\label{epi}
\begin{cases}
\min\limits_{\bfc, \theta} & \theta \\
\ \st & \Big\| \frac{1}{N} \sum\limits_{i=1}^N \big( f(x_i; \bfc) - y_i \big) \Big\|_{\infty} \le \theta.
\end{cases}
\ee
Denote by $\vartheta_{\min}$ the minimum value of \reff{epi}. 
In this paper, we aim at computing the global minimum value $\vartheta_{\min}$ 
and finding a global minimizer $(\bfc^*, \theta^*)$.

In this work, we propose a new approach for training polynomial activation functions for deep neural networks.
Our main contributions are:
\bit
\item We propose a training framework for selecting polynomial activation functions whose coefficients are to be learned.

\item We formulate the neural network training as a polynomial optimization problem and solve it by the Moment-SOS hierarchy.

\item We provide numerical experiments that demonstrate the efficiency and accuracy of learning activation function in presence of noises.
For the noiseless case, our method can recover activation functions exactly.

\eit

This paper is organized as follows.
Some basics on polynomial optimization are reviewed in Section~\ref{pre}.
In Section~\ref{ms_hier}, we give the Moment-SOS hierarchy of optimization problem \reff{epi} 
and prove convergence under certain assumptions.
Numerical experiments are presented in Section~\ref{eg}.
Some conclusions are drawn in Section~\ref{conclu}.

\section{Background}
\label{pre}

Denote by $\N$ the set of nonnegative integers and $\re$ the real field.
For a positive integer $m$, denote $[m] \coloneqq \{1\ddd m\}$.
For a polynomial $p$, $\deg(p)$ denotes the total degree.
The ring of polynomials in $z = (z_1 \ddd z_n)$ with real coefficients is denoted as  $\re[z]$.
For a degree $d$, the subset of polynomials in $\re[z]$ with degrees at most $d$ is denoted as $\re[z]_d$.
For $p \in \re[z]$, $\nabla p$ denotes its gradient with respect to $z$,
and $\nabla^2 p$ denotes its Hessian with respect to $z$.
Given a power $\af = (\af_1, \ldots, \af_n)$, we
denote the monomial $z^\af \coloneqq z_1^{\af_1} \cdots z_n^{\af_n}$ and $|\af| \coloneqq \af_1 + \cdots + \af_n$.
Denote by $e_i$ the standard basis vector such that the $i$th entry is $1$ and $0$ otherwise.
The infinity norm of $z$ is
$\|z\|_{\infty} \coloneqq \max_{1 \le i \le n} |z_i|$.
The Euclidean norm of $z$ is
$\|z\|_2 \coloneqq (|z_1|^2 + \cdots + |z_n|^2)^{1/2}$.
For a symmetric matrix $X$, $X \succeq 0$
means $X$ is positive semidefinite.

\smallskip
A polynomial $\sigma \in \re[z]$ is said to be a {\it sum of squares} (SOS)
if there are polynomials $p_1, \ldots, p_s \in \re[z]$ such that
$\sigma = p_1^2 + \cdots + p_s^2$.
For example, the polynomial
\[
z_1^4 - z_1^2(1+2z_1z_2) + z_1^2z_2^2 + 2z_1z_2 + 1 = z_1^2 + (1-z_1^2+z_1z_2)^2
\]
is SOS.
The cone of all SOS polynomials in $\re[z]$ is denoted as
$\Sigma[z]$.
For a tuple $g \coloneqq (g_1 \ddd g_s)$
of polynomials in $\re[z]$,
the degree-$2k$ truncation of its quadratic module is the set ($g_0 \coloneqq 1$)
\[
\qm[g]_{2k} \coloneqq \Big\{ \sum_{j=0}^{s} \sig_j g_{j} :
\sig_j \in \Sig[z], \deg(\sig_j g_{j}) \le 2k \Big\}.
\]
The set $\qm[g]_{2k}$ is convex and can be represented by semidefinite programming.
We refer to Section~2.5 of \cite{nie2023moment} for more details.

\smallskip
For a power $\af = (\af_1 \ddd \af_n)$, let
$\N_d^n \coloneqq\{ \af \in \N^n : |\af| \le d \}$
be the set of monomial powers with degrees at most $d$.
The symbol $\re^{\N^n_d}$ denotes the space of all real vectors labeled by $\af \in \N_d^n$.
A vector $w$ in $\re^{\N_{d}^{n}}$
is said to be a truncated multi-sequence (tms) of degree $d$.
For a given $w \in \re^{\N_d^{n}}$,
the {\it Riesz functional} generated by $w$
is the linear functional $\mathscr{R}_{w}$ acting on $\re[z]_d$ such that
\[
\mathscr{R}_w \Big( \sum_{\af \in \N_d^n} p_{\af} z^{\af} \Big) \coloneqq \sum_{\af \in \N_d^n} p_{\af} w_{\af}.
\]
For convenience, we denote
$\langle p \, , \, w \rangle \coloneq \mathscr{R}_w(p)$ for $p \in \re[z]_d$.
The {\it localizing matrix} of $p \in \re[z]$ generated by $w$ is
\[
L_{p}^{(k)}[w] \coloneqq \mathscr{R}_{w} \left( p(z) \cdot [z]_{k_1}[z]^T_{k_1} \right),
\quad k_1 \coloneqq \lfloor k - \deg(p) / 2 \rfloor.
\]
For instance, when $n=2, k=2$, and $p = 1-z_1^2-z_2^2$,
\[
L^{(2)}_p[w] =
\bbm 
w_{00} - w_{20} - w_{02} & w_{10} - w_{30} - w_{12} & w_{01} - w_{21} - w_{03} \\
w_{10} - w_{30} - w_{12} & w_{20} - w_{40} - w_{22}  & w_{11} - w_{31} - w_{13} \\
w_{01} - w_{21} - w_{03} & w_{11} - w_{31} - w_{13} & w_{02} - w_{22} - w_{04}
\ebm.
\]
In particular, if $p=1$ is the constant polynomial, we get the {\it moment matrix}
\[
M_k[w] \coloneqq L_{1}^{(k)}[w].
\]
We refer to \cite{lasserre2015introduction,Lau09,nie2023moment}
for more detailed introductions to polynomial optimization.

\section{The Moment-SOS Hierarchy}
\label{ms_hier}

In this section, we give a hierarchy of Moment-SOS relaxations to solve \reff{epi},
and discuss how to extract minimizers from the moment relaxations.
For convenience of notation, we denote ($\bfc$ is defined in \reff{c})
\[
z \coloneqq (\bfc, \theta) = (z_1 \ddd z_n), \quad n = d_1 + \cdots + d_{D} + 2.
\]
We rewrite the optimization problem \reff{epi} as follows:
\be\label{epi_hat}
\begin{cases}
\min\limits_{z \in \re^n} & z_n \\
\ \st & \| \hat{f}(z) - \hat{y} \|_{\infty} \le z_n.
\end{cases}
\ee
In the above, 
$\hat{f}(z) \coloneqq \frac{1}{N} \sum_{i=1}^N f(x_i; \bfc)$ and
$\hat{y} \coloneqq \frac{1}{N} \sum_{i=1}^N y_i$.
Observe that 
\[
\| \hat{f}(z) - \hat{y} \|_{\infty} \le z_n \iff
-z_n \le \hat{f}_j(z) - \hat{y}_j \le z_n, \quad j = 1 \ddd m_{D+1}.
\]
Define the polynomial tuple $g(z) \coloneqq \big( g_1(z), \ddd g_{\hat{m}}(z) \big)$ such that
\begin{align*}
g(z) &= \Big( z_n + (\hat{f}_1(z) - \hat{y}_1 ),  z_n - (\hat{f}_1(z) - \hat{y}_1 ) \ddd \\
&\qquad z_n + (\hat{f}_{m_{D+1}}(z) - \hat{y}_{m_{D+1}} ), z_n - (\hat{f}_{m_{D+1}}(z) - \hat{y}_{m_{D+1}} ) \Big),
\end{align*}
where $\hat{m} = 2m_{D+1}$.
Then, \reff{epi_hat} is equivalent to
\be\label{pop}
\begin{cases}
\min\limits_{z \in \re^n} & z_n \\
\ \st &  g_1(z) \ge 0 \ddd g_{\hat{m}}(z) \ge 0.
\end{cases}
\ee
The feasible set of \reff{pop} is
\[
K = \{ z \in \re^{n} \, : \, g_1(z) \ge 0 \ddd g_{\hat{m}}(z) \ge 0  \}.
\]
Let
$
k_0 \coloneqq  \lceil \deg(g)/2\rceil  .
$
For a degree $k\ge k_0$, the $k$th order SOS-relaxation for \reff{pop} is
\be\label{pop_sos}
\left\{  \baray{rrl}
\vartheta_{sos,k} \coloneqq &\max&  \gm \\
& \st  &  z_n - \gm \in \qmod{g}_{2k}.
\earay \right.
\ee
Its dual optimization problem is the $k$th order moment relaxation
\be\label{pop_mom}
\left\{  \baray{rrl}
\vartheta_{mom,k} \coloneqq &\min\limits_{w} & \langle z_n \, , \, w \rangle \\
& \st & L_{g_j}^{(k)}[w] \succeq 0, \, j = 1 \ddd \hat{m} , \\
& & M_k[w] \succeq 0, \, w_0 = 1, \, w \in \re^{\N^{n}_{2k}}.
\earay \right.
\ee
Note that $\langle z_n \, , \, w \rangle = \mathscr{R}_w(z_n) = w_{e_n}$, 
the entry of $w$ indexed by $e_n$.
We refer to Section~\ref{pre} for the above notation.

In computational practice, we start with relaxation order $k=k_0$.
If the relaxation \reff{pop_mom} is successfully solved, we extract its minimizers.
Otherwise, we increase $k$ by one and repeat this process.
For $k = k_0, k_0+1 \ddd$ the sequence of primal-dual pairs \reff{pop_sos}-\reff{pop_mom} is called the Moment-SOS hierarchy. 
Recall that $\vartheta_{\min}$ denotes the minimum value of \reff{epi}.
Since \reff{epi} can be equivalently formulated as \reff{epi_hat} and \reff{pop}, $\vartheta_{\min}$ is also the minimum value for \reff{epi_hat} and \reff{pop}.
For each relaxation order $k$, it holds that
\be \label{bound}
\vartheta_{sos,k} \le \vartheta_{mom,k} \le \vartheta_{\min}. 
\ee
Moreover, both sequences $\{\vartheta_{sos,k}\}_{k=1}^{\infty}$ and $\{\vartheta_{mom,k}\}_{k=1}^{\infty}$ are monotonically increasing.
If the quadratic module $\qmod{g}$ is archimedean \cite[Section~2.6]{nie2023moment}, then 
the Moment-SOS hierarchy of \reff{pop_sos}-\reff{pop_mom} has asymptotic convergence, i.e.,
\[
\lim_{k\to \infty} \vartheta_{sos,k} = \lim_{k\to \infty} \vartheta_{mom,k} = \vartheta_{\min}.
\]
Under some conditions, the Moment-SOS hierarchy has finite convergence,
and we can extract minimizers of \reff{epi_hat} from the moment relaxation \reff{pop_mom}.
Suppose $w^*$ is a minimizer of \reff{pop_mom} for a relaxation order $k \ge k_0$.
To obtain a minimizer of \reff{epi_hat}, we typically need the {\it flat truncation} condition:
there exists an integer $d \in [k_0, k]$ such that 
\be\label{flat}
r \coloneqq \rank \, M_d[w^*] = \rank \, M_{d-k_0}[w^*].
\ee

\begin{theorem}\cite{nie2023moment}
Suppose $w^*$ is a minimizer of \reff{pop_mom}.
If the flat truncation condition \reff{flat} holds, 
then the Moment-SOS relaxation is tight, 
i.e., $\vartheta_{\min} = \vartheta_{sos,k} = \vartheta_{mom,k}$.
\end{theorem}

For the case $r=1$, if $w^*$ is a minimizer of \reff{pop_mom}, 
then $z^* = (w^*_{e_1} \ddd w^*_{e_n})$ is a minimizer of \reff{epi_hat}. 
For the case $r>1$, we refer to Section~5.3 of \cite{nie2023moment} for the extraction of minimizers.

There is a close relationship between the finite convergence theory of Moment-SOS hierarchy and local optimality conditions.
Suppose  $u$ is a local minimizer of \reff{pop},
and there exists a Lagrange multiplier
$
\lmd = ( \lmd_1 \ddd \lmd_{\hat{m}} )
$
satisfying the Karush-Kuhn-Tucker (KKT) conditions
\be\label{fooc}
\nabla  z_n = e_n = \lmd_1 \nabla g_1(u) + \cdots + \lmd_{\hat{m}} \nabla g_{\hat{m}}(u),
\ee
\be\label{cc}
g_j(u) \ge 0, \, \lmd_j \ge 0, \, \lmd_j g_j(u) = 0, \, j = 1 \ddd \hat{m}.
\ee
For $\lmd$ satisfying \reff{fooc}-\reff{cc}, the Lagrangian function is
\be\label{lagrangian}
\mathscr{L}(z) \coloneqq z_n - \lmd_1 g_1(z) - \cdots - \lmd_{\hat{m}} g_{\hat{m}}(z).
\ee
Recall that a polynomial $\varphi(z)$ is SOS-convex 
if its Hessian $\nabla^2 \varphi(z) = H(z)H(z)^T$ 
for some matrix polynomial $H(z)$.
We refer to Chapter~7 of \cite{nie2023moment} for SOS convex polynomials.

\begin{theorem}
Let $\mathscr{L}(z)$ be the Lagrangian function for \reff{pop}. 
Suppose $\mathscr{L}(z)$ is SOS-convex, then the Moment-SOS relaxation is tight for all $k\ge k_0$, 
i.e., 
$
\vartheta_{sos,k} = \vartheta_{mom,k} = \vartheta_{\min}.
$
\end{theorem}

\begin{proof}
The Lagrangian function for \reff{pop} is \reff{lagrangian} with
Lagrange multipliers $\lmd_1 \ddd \lmd_{\hat{m}} \ge 0$.
Suppose $z^*$ is the minimizer of \reff{pop}.
Then, we have $\mathscr{L}(z^*) = \vartheta_{\min}$ and $\nabla \mathscr{L}(z^*) = 0$.
By the integration formula,
\begin{align*}
\mathscr{L}(z) &= \mathscr{L}(z^*) + (z-z^*)\nabla\mathscr{L}(z^*)
+ \int_0^1 \int_0^t (z-z^*)^T \nabla^2 \mathscr{L}(s(z-z^*)+z^*) (z-z^*) \, \mt{d}s\mt{d}t \\
&= \vartheta_{\min} + (z-z^*)^T \Big( \int_0^1 \int_0^t  \nabla^2 \mathscr{L}(s(z-z^*)+z^*)  \, \mt{d}s\mt{d}t \Big) (z-z^*).
\end{align*}
Since $\mathscr{L}(z)$ is SOS-convex, the Hessian $\nabla^2 \mathscr{L}(z)$ is an SOS matrix polynomial, so $\sig(z) \coloneqq \mathscr{L}(z) - \vartheta_{\min}$ is an SOS polynomial.
This implies
\[
z_n - \vartheta_{\min} = \sig(z) + \lmd_1 g_1(z) + \cdots + \lmd_{\hat{m}} g_{\hat{m}}(z)
\in \qmod{g}_{2k_0}.
\]
Therefore, $\vartheta_{sos,k_0} \ge \vartheta_{\min}$. 
By \reff{bound} and the monotonicity of $\vartheta_{sos,k}$ and $\vartheta_{mom,k}$, 
we get the desired result.
\end{proof}

\begin{remark}
If the minimum value for \reff{pop_mom} is zero, then the coefficient recovery for the activation functions is exact.
\end{remark}

\begin{example}\rm

Consider a two-hidden-layer network with polynomial activation functions
\[
p_1(t) = c_{1,2}t^2 + c_{1,1}t + 1, \quad p_2(t) = c_{2,1}t + c_{2,0}.
\]
Let $m_0 = m_1 = m_2 = m_3 = 4$ and weight matrices
\[
W_1 = 
\bbm
 1   &  0  &  -1   &  1 \\
0  &   1   &  1   &  1 \\
-1  &   0  &   1  &  -1 \\
-2  &   1  &  -1  &   0
\ebm,
\quad W_2 = 
\bbm
1  &  -1  &   0   &  2 \\
2   &  1   &  1   &  0 \\
1   &  1   &  1  &   2 \\
0  &   1  &   1   &  1 
\ebm,
\quad W_3 = 
\bbm  
1  &    1   &   0  &    1 \\
-2  &   -1  &    1   &   1 \\
1  &    0   &   1   &   1 \\
-1  &    0  &   -2   &   1
\ebm.
\]
Suppose the sample size $N=2$ and let $\{(x_i, y_i)\}_{i=1}^N$ be the training data such that
\[
x_1 = \bbm 2 \\ 1 \\ 0 \\ -1 \ebm, \quad y_1 = \bbm 66 \\ -22 \\ 106 \\ -104 \ebm, 
\quad x_2 = \bbm -1 \\ 1 \\ 1 \\ 1 \ebm, \quad y_2 = \bbm 38 \\ 30 \\ 46 \\ -33 \ebm.
\]
For $z \coloneqq (\bfc, \theta) = (c_{1,1}, c_{1,2}, c_{2,0}, c_{2,1}, \theta)$, we have
\[
\hat{f}(z) - \hat{y} = 
 \bbm
3c_{2,0} + 9c_{2,1} + 29c_{1,2}c_{2,1} - 52 \\
5c_{1,1}c_{2,1} - c_{2,0} + 5c_{1,2}c_{2,1} - 4 \\
3c_{2,0} + 10c_{2,1} - c_{1,1}c_{2,1} + 41c_{1,2}c_{2,1} - 76 \\
\frac{5}{2}c_{1,1}c_{2,1} - 9c_{2,1} - 2c_{2,0} - \frac{73}{2}c_{1,2}c_{2,1} + \frac{137}{2} \\
\ebm.
\]
To solve the optimization problem \reff{epi_hat}, we solve the moment relaxation \reff{pop_mom}.
For relaxation order $k=1$, the flat truncation condition \reff{flat} does not hold.
However, \reff{flat} holds for $k=2$, and we can extract the minimizer for \reff{epi_hat}:
\[
\bfc^* = (1, -2, 1, -1), \quad \theta^* = 0.
\]

\end{example}

\section{Numerical experiments}
\label{eg}

This section provides numerical experiments for learning polynomial activation functions.
All computations are implemented using MATLAB R2023b on a MacBook Pro equipped with Apple M1 Max processor and 16GB RAM. The semidefinite programs are solved by the software \texttt{GloptiPoly} \cite{henrion2009gloptipoly}, which calls the SDP package \texttt{SeDuMi} \cite{sturm1999using}. For neatness, only four decimal digits are displayed for computational results.

{\bf Normalization of coefficients.}
We remark that the normalization of coefficients is needed.
To see this, consider a two-hidden-layer network with linear activation functions
\[
p_1 = c_{1,1}t + c_{1,0}, \quad p_2(t) = c_{2,1}t + c_{2,0}.
\]
Given any weight matrices $W_1,W_2,W_3$,
for an input $x \in \re^{m_0}$ , \reff{f_map} gives the output
\begin{equation*}
f(x; c_{1,0}, c_{1,1}, c_{2,0}, c_{2,1}) = c_{2,1}(c_{1,1}W_3W_2W_1x+c_{1,0}W_3W_2\mathbf{1})+c_{2,0}W_3\mathbf{1},
\end{equation*}
where $\mathbf{1}=[1 \cdots 1]^T \in \re^{m_3}$. 
Note that 
\[
f(x; c_{1,0}, c_{1,1}, c_{2,0}, c_{2,1}) = f(x; \tau c_{1,0}, \tau c_{1,1}, c_{2,0}, \tau^{-1} c_{2,1})
\]
for all scalar $\tau \ne 0$.
So in general, we can normalize the coefficients such that $c_{1,0}=1$.
For $D > 2$ hidden layers, we can similarly normalize the coefficients such that $c_{\ell,0}=1$ for all $\ell = 1 \ddd D-1$.

{\bf Experiments on synthetic data.} 
We illustrate our method on synthetic $D$-layer networks.
For $\ell = 1 \ddd D+1$, let
$
W_{\ell} \in \re^{m_{\ell} \times m_{\ell - 1}}
$
be the randomly generated weight matrices, and let
\[
p_{\ell}(t) = c_{\ell, d_{\ell}}t^{d_{\ell}} + \cdots + c_{\ell, 1}t + c_{\ell, 0}, 
\quad \ell = 1 \ddd D,
\]
be the polynomial activation functions.
In the above, we normalize the coefficients such that $c_{1, 0} = \cdots = c_{D-1,0} = 1$.
We then randomly generate the rest of the coefficients
$\bfc$ (defined in \reff{c}), for which we denote $\bfc_{\mathrm{true}}$.
The training input-output data $\{(x_i, y_i)\}_{i=1}^N$ is generated as follows.
We choose the input $x_i$ randomly.
The corresponding output $y_i$ is obtained as 
($p_{\ell}$ is applied element-wise with the coefficient vector $\bfc_{\mathrm{true}}$)
\be\label{xi_to_yi}
h_0 \coloneqq x_i, \quad v_{\ell} \coloneqq W_{\ell}h_{\ell-1},
\quad h_{\ell} \coloneqq p_{\ell}(v_{\ell}),
\quad \ell = 1 \ddd D, \quad y_i = W_{D+1}h_D + \vareps,
\ee
where $\vareps = \texttt{1e-2*randn}(m_{D+1},1)$ is a random perturbing noise vector.
We learn the activation polynomial coefficient vector $\bfc$
by solving Moment-SOS relaxations \reff{pop_sos}-\reff{pop_mom}.
We remark that for all generated instances,
the learning problem \reff{epi_hat} is solved successfully by the relaxation \reff{pop_mom} with order $k=3$.
The learned coefficient vector $\bfc$ is denoted as $\bfc_{\mathrm{pred}}$.
The learning accuracy is assessed by the errors
\[
\mathrm{AbsErr} = \big\| \bfc_{\mathrm{true}} - \bfc_{\mathrm{pred}} \big\|_2, \quad 
\mathrm{RelErr} = \big\| \bfc_{\mathrm{true}} - \bfc_{\mathrm{pred}} \big\|_2 \, \big/ \, \| \vareps \|_2.
\]

{\it i)} For the two-layer network with quadratic activation functions
\[
p_1(t) = c_{1,2}t^2 + c_{1,1}t + 1, \quad p_2(t) = c_{2,2}t^2 + c_{2,1}t + c_{2,0},
\]
the errors and consumed time for different dimensions are displayed in Table~\ref{tab:2layer22}.

\begin{table}[htb]
\caption{Accuracy for $2$-layer DNN with quadratic activations.}
\label{tab:2layer22}
\centering
\begin{tabular}{ccccc}
\toprule
$(N, m_0, m_1, m_2, m_3)$ & $\mathrm{AbsErr}$ & $\mathrm{RelErr}$ & Time & $\| \vareps \|_2$ \\
\midrule
$(50, 20, 20, 20, 20)$ &    0.0019    &    0.0452    &    5.3287   &  0.0430   \\
$(100, 25, 25, 25, 25)$ &    0.0013    &    0.0295    &   5.9034   &   0.0485    \\
$(150, 30, 30, 30, 30)$ &    0.0014    &   0.0265     &   7.0377   &    0.0521   \\
$(200, 35, 35, 35, 35)$ &    0.0018    &    0.0323    &   6.3813   &   0.0557    \\
$(250, 40, 40, 40, 40)$ &    0.0012    &    0.0201    &   7.9470   &    0.0598   \\
\bottomrule
\end{tabular}
\end{table}

{\it ii)} For the two-layer network with quadratic and cubic activation functions
\[
p_1(t) = c_{1,2}t^2 + c_{1,1}t + 1, \quad p_2(t) = c_{2,3}t^3 + c_{2,2}t^2 + c_{2,1}t + c_{2,0},
\]
the errors and consumed time for different dimensions are displayed in Table~\ref{tab:2layer23}.

\begin{table}[htb]
\caption{Accuracy for $2$-layer DNN with quadratic and cubic activations.}
\label{tab:2layer23}
\centering
\begin{tabular}{ccccc}
\toprule
$(N, m_0, m_1, m_2, m_3)$ & $\mathrm{AbsErr}$ & $\mathrm{RelErr}$ & Time & $\| \vareps \|_2$ \\
\midrule
$(50, 20, 20, 20, 20)$ &    0.0065    &    0.1332    &    17.2626   &   0.0480   \\
$(100, 25, 25, 25, 25)$ &    0.0061    &    0.1203    &   20.7595   &   0.0502    \\
$(150, 30, 30, 30, 30)$ &    0.0020    &    0.0381    &   22.8920   &   0.0513   \\
$(200, 35, 35, 35, 35)$ &    0.0012    &    0.0197    &  24.5851    &   0.0579    \\
$(250, 40, 40, 40, 40)$ &     0.0012   &    0.0182    &   26.3780   &   0.0609    \\
\bottomrule
\end{tabular}
\end{table}

{\it iii)} For the three-layer network with linear activation functions
\[
p_{1}(t) = c_{1,1}t + 1, \quad p_{2}(t) = c_{2,1}t + 1, \quad p_{3}(t) = c_{3,1}t + c_{3,0},
\]
the errors and consumed time for different dimensions are displayed in Table~\ref{tab:3layer111}.

\begin{table}[htb]
\caption{Accuracy for $3$-layer DNN with linear activations.}
\label{tab:3layer111}
\centering
\begin{tabular}{ccccc}
\toprule
$(N, m_0, m_1, m_2, m_3, m_4)$ & $\mathrm{AbsErr}$ & $\mathrm{RelErr}$ & Time & $\| \vareps \|_2$ \\
\midrule
$(50, 20, 20, 20, 20, 20)$ &    0.0016    &    0.0368    &  0.9970   &   0.0442    \\
$(100, 25, 25, 25, 25, 25)$ &    0.0012   &    0.0254     &   0.9767   &    0.0475   \\
$(150, 30, 30, 30, 30, 30)$ &    0.0006    &    0.0114    &   1.1439   &   0.0549    \\
$(200, 35, 35, 35, 35, 35)$ &    0.0011    &    0.0172    &   1.2455   &   0.0651    \\
$(250, 40, 40, 40, 40, 40)$ &    0.0009    &   0.0151     &   1.3911   &   0.0598    \\
\bottomrule
\end{tabular}
\end{table}

{\it iv)} For the three-layer network with linear and quadratic activation functions
\[
p_{1}(t) = c_{1,2}t^2 + c_{1,1}t + 1, \quad p_{2}(t) = c_{2,1}t + 1, \quad 
p_{3}(t) = c_{3,1}t + c_{3,0},
\]
the errors and consumed time for different dimensions are displayed in Table~\ref{tab:3layer211}.

\begin{table}[htb]
\caption{Accuracy for $3$-layer DNN with linear and quadratic activations.}
\label{tab:3layer211}
\centering
\begin{tabular}{ccccc}
\toprule
$(N, m_0, m_1, m_2, m_3, m_4)$ & $\mathrm{AbsErr}$ & $\mathrm{RelErr}$ & Time & $\| \vareps \|_2$ \\
\midrule
$(50, 20, 20, 20, 20, 20)$ &    0.0012    &   0.0295    &   2.8547   &    0.0422   \\
$(100, 25, 25, 25, 25, 25)$ &   0.0055    &    0.1252     &   3.4459  &    0.0484   \\
$(150, 30, 30, 30, 30, 30)$ &    0.0030    &    0.0540    &   3.8279   &    0.0553   \\
$(200, 35, 35, 35, 35, 35)$ &    0.0009    &   0.0143     &   4.0465   &    0.0575   \\
$(250, 40, 40, 40, 40, 40)$ &    0.0007    &    0.0122    &   4.3156   &   0.0628    \\
\bottomrule
\end{tabular}
\end{table}

{\bf Residual analysis for two-layer network.}
We consider a two-hidden-layer network with randomly generated weight matrices
$W_1, W_2, W_3$.
Suppose the polynomial activation functions are
\[
p_1(t) = c_{1,3}t^3 + c_{1,2}t^2 + c_{1,1}t + 1, \quad p_2(t) = c_{2,2}t^2 + c_{2,1}t + c_{2,0}.
\]
We randomly generate $\bfc \coloneqq (c_{1,1}, c_{1,2}, c_{1,3}, c_{2,0}, c_{2,1}, c_{2,2})$, for which we denote $\bfc_{\mathrm{true}}$.
The training input-output data $\{(x_i, y_i)\}_{i=1}^N$ is generated as follows.
We choose the input $x_i$ randomly,
and the corresponding output $y_i$ is obtained as in \reff{xi_to_yi} with the coefficient vector $\bfc_{\mathrm{true}}$.
We learn the activation polynomial coefficient vector $\bfc$
by solving Moment-SOS relaxations \reff{pop_sos}-\reff{pop_mom}.
The learned coefficient vector $\bfc$ is denoted as $\bfc_{\mathrm{pred}}$.
To evaluate the quality of the learned $\bfc_{\mathrm{pred}}$, 
we generate an independent test set $\{(\tilde{x}_i, \tilde{y}_i)\}_{i=1}^N$ 
in the same way as the training set.
The predicted output $y_i'$ is obtained as in \reff{xi_to_yi} with the coefficient vector $\bfc_{\mathrm{pred}}$ and $\vareps = 0$.
We use two residual plots to illustrate the learning accuracy.
First, for each $j=1 \ddd m_3$, 
we plot the $j$th entry of the residual $\vareps_i \coloneqq y'_i - \tilde{y}_i$ against the index $i = 1 \ddd N$ (the left plot in Figure~\ref{fig:two-layer-residual}).
Second, we plot the norm $\|\vareps_i\|_2$ against the index $i$ (the right plot in Figure~\ref{fig:two-layer-residual}).
For neatness of the paper, 
we display these two residual plots only for $(N,m_0,m_1,m_2,m_3)=(50,20,20,20,20)$ in Figure~\ref{fig:two-layer-residual}.
We can see that in both residual plots, 
the points are close to zeros and show no visible trends along the index $i$.

\begin{figure}[htb]
\centering
\includegraphics[width=0.45\textwidth]{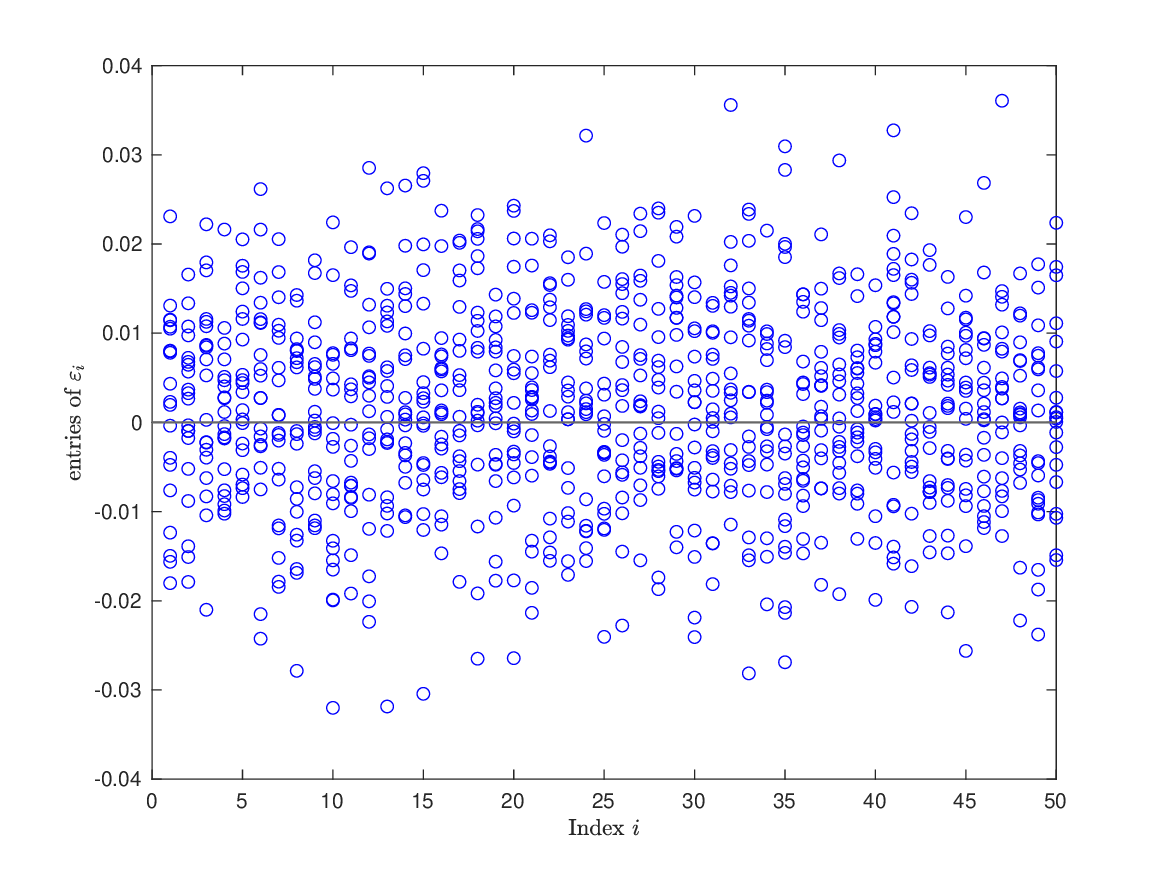}
\includegraphics[width=0.45\textwidth]{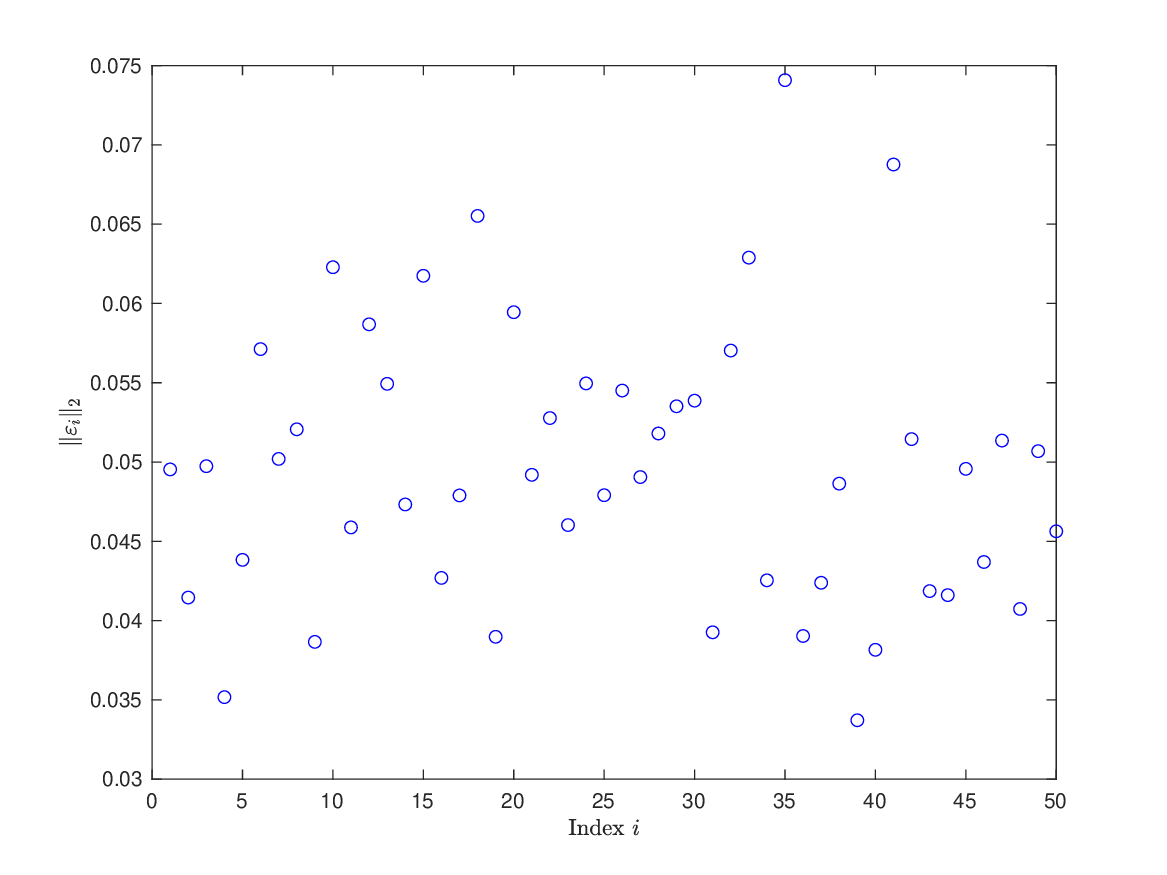}
\caption{Plots for residual error $\vareps_i$ with $i=1 \ddd 50$.}
\label{fig:two-layer-residual}
\end{figure}

For different $(N,m_0,m_1,m_2,m_3)$, we report in Table~\ref{tab:two-layer-residual} the mean squared error (MSE) and root mean squared error (RMSE):
\[
\mathrm{MSE} = \frac{1}{N} \sum_{i=1}^{N} \|\vareps_i\|_2^2,
\quad \mathrm{RMSE} = \sqrt{\mathrm{MSE}}.
\]

\begin{table}[htb]
\caption{Test errors for $2$-layer DNN with cubic and quadratic activations.}
\label{tab:two-layer-residual}
\centering
\begin{tabular}{ccc}
\toprule
$(N, m_0, m_1, m_2, m_3)$ & MSE & RMSE \\
\midrule
$(50, 20, 20, 20, 20)$ &   0.0090    &    0.0830    \\
$(100, 25, 25, 25, 25)$ &    0.0026   &    0.0511   \\
$(150, 30, 30, 30, 30)$ &     0.0046   &   0.0669   \\
$(200, 35, 35, 35, 35)$ &    0.0078    &    0.0808   \\
$(250, 40, 40, 40, 40)$ &    0.0133    &   0.1002   \\  
\bottomrule
\end{tabular}
\end{table}

\section{Concluding remarks}
\label{conclu}

This paper proposes a new training framework for deep neural networks with learnable polynomial activation functions. 
This training process can be formulated as polynomial optimization,
which is solvable by Moment-SOS relaxations.
Numerical experiments demonstrate that our approach can learn activation polynomial coefficients exactly in absence of noises and remains robust under noises. 
This highlights the potential of polynomial optimization techniques in deep learning, 
providing new insights into the design and analysis of activation functions.

The Moment-SOS hierarchy is a powerful analytical tool for studying properties of deep neural networks. 
It provides globally optimal performance benchmarks and structural insights into the interplay between network weights and polynomial activation functions. 
The global optimality certification provided by the Moment-SOS hierarchy gives a promising direction to investigate hybrid methods, which combine gradient-based optimization methods.  
For example, backpropagation, with its unparalleled scalability, could be used for the bulk of the training, while the Moment-SOS hierarchy could be used to solve critical subproblems. 
Moreover, it could be applied to a single layer to find a globally optimal polynomial activation functions, which can be further used to initialize larger networks. 
It could also be used to design regularization terms that guide the gradient descent process toward a better basin of attraction.

\begin{ack}
%
%

Jiawang Nie and Linghao Zhang are partially supported by the National Science Foundation grant DMS-2110780.

\end{ack}

\end{document}